\documentclass[12pt]{amsart}
\usepackage[T1]{fontenc}
\usepackage{newtxtext,dsfont,amssymb,amsmath,mathrsfs,fourier,baskervald}

\usepackage[text={7in,9in},centering]{geometry}

\usepackage[all]{xy}
\usepackage[dvipsnames]{xcolor}   
\usepackage{xparse}
\usepackage{xr-hyper}
\usepackage[linktocpage=true,colorlinks=true,hyperindex,citecolor=teal,linkcolor=blue]{hyperref}

\newtheorem{theorem}{Theorem}[section] 
\newtheorem{proposition}[theorem]{Proposition}

\newtheorem{corollary}[theorem]{Corollary}
\theoremstyle{definition}
\newtheorem{definition}[theorem]{Definition}
\newtheorem{example}[theorem]{Example} 
\newtheorem{remark}[theorem]{Remark} 
\numberwithin{equation}{section}

\def\:{\colon}
\newcommand{\End}[2][]{\operatorname{End}_{#1}(#2)}
\def\P#1{\mathcal{P}(#1)}
\newcommand{\Idl}[1]{\operatorname{Idl}(#1)}
\renewcommand{\Im}[1]{\operatorname{Im}(#1)}
\newcommand{\id}[1]{\operatorname{id}_{#1}}
\def\set#1#2{\left\{{#1}\left.\right|\,{#2}\right\}}
\newcommand{\Sub}[1]{\operatorname{Sub}(#1)}
\newcommand{\cont}{\subseteq}
                  	    
\begin{document}  
\title{Primitive quantales}

\author{Amartya Goswami}
\address{[1] Department of Mathematics and Applied Mathematics, University of Johannesburg, P.O. Box 524, Auckland Park 2006, South Africa.
[2] National Institute for Theoretical and Computational Sciences (NITheCS), South Africa.} 

\email{agoswami@uj.ac.za}

\author{Elena Caviglia}

\address{[1] Department of Mathematical Sciences, Stellenbosch University, South Africa. [2]  National Institute for Theoretical and Computational Sciences (NITheCS), South Africa.}
\email{elena.caviglia@outlook.com}

\author{Luca Mesiti}
\address{Department of Mathematical Sciences, Stellenbosch University, South Africa.}

\email{luca.mesiti@outlook.com}

\begin{abstract}
We generalize Jacobson's notion of primitive ring to the setting of quantales. We show that every primitive ring gives rise to a primitive quantale of ideals. We then prove a density theorem for strongly primitive quantales. Furthermore, we show that primitive quantales are prime and commutative strongly primitive quantales are field quantales.
\end{abstract}  




 
\maketitle 
   
\section{Introduction}

It is well known (see \cite{J45} and \cite{J56}) that in the theory of noncommutative rings, the notion of primitive ideals plays a crucial role in determining the structure of the ring. Furthermore, in \cite{J56}, a hull-kernel-type topology is introduced on the set of all primitive ideals of a ring, and representations of biregular rings are studied therein. It is well-known that an effective way to develop the theory of primitive ideals is to study the theory of primitive rings. Indeed, there is a correspondence between primitive ideals and primitive rings. It is worth mentioning that primitive ideals and primitive rings have demonstrated their profound importance in understanding various structural aspects of modules \cite{J56, R88}, Lie algebras \cite{KPP12}, enveloping algebras \cite{D96, J83}, PI-algebras \cite{J75}, quantum groups \cite{J95} and skew polynomial rings \cite{I79} among others.

Analogous to Dilworth’s initiation of abstract ideal theory in \cite{Dil62} (see also \cite{And74}), one may naturally inquire about the possibility of abstracting Jacobson’s theory of primitive rings to the setting of lattices. In this context, we believe that the most suitable lattice-theoretic framework is provided by Mulvey’s quantales, as introduced in \cite{Mul86}. The aim of our investigation is to start addressing this question.

In this paper, we generalize Jacobson's notion of primitive ring to the setting of quantales. One of the main challenges in such a setting is the lack of differences, in the sense that the join operation of a quantale has no inverses. In \cite{KN14}, the authors face a similar challenge when extending the notion of primitive ring to the setting of hemirings. They address the lack of differences by considering the ring of differences associated to a hemiring. And they consequently present a density theorem that involves the ring of differences and reduces to use the classical Jacobson Density Theorem. In this paper, we explore what can be done in the setting of quantales despite the lack of differences.

Let us briefly summarize the content of this paper. In Section \ref{moq}, we recall the necessary definitions and fundamental facts concerning modules over quantales and set the stage for the introduction of primitive quantales. We establish several results that may be viewed as extensions of their classical ring-theoretic analogues. Section \ref{pqt} is devoted to defining primitive quantales and introducing the necessary concepts to formulate and prove a density theorem, which is presented in Section \ref{dtq}. Among other results, we show that, for quantales, primitivity implies primeness, and that commutative strongly primitive quantales are field quantales.

\section{Modules over a quantale}\label{moq}

Recall that a \emph{(unital) quantale} is a complete lattice $Q$ endowed with a product operation $\ast\:Q\times Q\to Q$ and an identity element $e\in Q$ such that
\begin{enumerate}
\item $\ast$ distributes over arbitrary joins in both components;
        
\item $\ast$ is associative;
        
\item for every $\alpha\in Q$ we have $\alpha\ast e=\alpha=e\ast \alpha$.
\end{enumerate}
In this work,     we shall only consider quantales that are non-trivial, \textit{i.e.},  different from $\{\bot\}$.

A quantale is called \emph{commutative} if its product $\ast$ is commutative.

A \emph{homomorphism of quantales} between quantales $Q$ and $Q'$ is a function $h\:Q\to Q'$ that preserves arbitrary joins in both variables, products, and identity elements.
A \emph{subquantale} of a quantale $Q$ is a subset $A$ of $Q$ containing the identity element $e$ that is closed under arbitrary joins and products.

\begin{remark}
Subquantales are equivalently given by injective homomorphisms of quantales. 

Moreover, the image $\Im{h}$ of a homomorphism of quantales $h\:Q\to Q'$ is a subquantale of $Q'$. Indeed, the image $\Im{h}$ is certainly a subset of $Q'$ containing $e_{Q'}$, since $h(e_Q)=e_{Q'}$. It is closed under arbitrary joins and products because $h$ preserves arbitrary joins and products.
\end{remark}

\begin{definition}[\cite{Ros94}]
    A \emph{module over a quantale $Q$} (also called a \emph{$Q$-module}) is a complete lattice $M$ equipped with a function $\cdot\:Q\times M\to M$ such that
\begin{enumerate}
\item $\cdot$ preserves arbitrary joins in both components;
        
\item for every $\alpha,\beta\in Q$ and every $m\in M$, we have $(\alpha \ast \beta)\cdot m=\alpha\cdot (\beta\cdot m)$;
        
\item for every $m\in M$, we have $e_Q\cdot m=m$.
\end{enumerate}
\end{definition}

Every complete lattice has a bottom element, given by the empty join. So every quantale and every module over a quantale have a bottom element.

Notice that, in every $Q$-module $M$, for every $m\in M$, we need to have $\bot_Q\cdot m=\bot_M$. Moreover, for every $\alpha\in Q$, $\alpha\cdot \bot_M=\bot_M$. Both are given by the condition that $\cdot$ preserves joins.
\begin{remark}
\label{remend}
Let $N$ be a complete lattice. The set $\End{N}$ of (arbitrary) join-preserving functions from $N$ to $N$ can be endowed with the structure of a quantale. Indeed, we can form joins of functions $N\to N$ pointwise, \textit{i.e.}, $(\bigvee f_i)(n):=\bigvee (f_i(n))$; and every suplattice is a complete lattice. Then composition gives an associative product operation $\ast\:\End{M}\times \End{M}\to \End{M}$, and the identity of $M$ is an identity element for such a product. It is straightforward to prove that $\End{M}$ is a quantale, using that the elements of $\End{M}$ are join-preserving functions and that joins in $\End{M}$ are defined pointwise.
\end{remark}

\begin{proposition}
Let $Q$ be a quantale. A module over $Q$ is equivalently given by a complete lattice $M$ and a function $[\cdot]\:Q\to \End{M}$ such that
\begin{enumerate}
\item $[\cdot]$ preserves arbitrary joins;

\item for every $\alpha,\beta\in Q$, we have $[\cdot](\alpha\ast \beta)=[\cdot](\alpha)\circ [\cdot](\beta)$;

\item $[\cdot](e_Q)=\id{M}$.
\end{enumerate}
That is, a module over $Q$ is equivalently given by a complete lattice $M$ and a homomorphism of quantales $[\cdot]\:Q\to \End{M}$.
\end{proposition}

\begin{proof}
Here is a sketch of the proof. Given $\cdot\:Q\times M\to M$, define $[\cdot]\:Q\to \End{M}$ to send $\alpha$ to the join-preserving function $m\mapsto \alpha\cdot m$. Conversely, given $[\cdot]$, define $\alpha\cdot m:=[\cdot](\alpha)(m)$.
\end{proof}

\begin{example}\label{fex}
\begin{enumerate}
\item[]

\item Every quantale $Q$ is a module over $Q$, with scalar multiplication $Q\times Q\to Q$ given by the product. It is easy to see that the three axioms of a $Q$-module  are given by the three axioms of a quantale.

\item \label{exaeverylatticeismodule}
Every complete lattice $N$ can be endowed with a structure of a module over the quantale $\{\bot,\top\}$ of truth values. Indeed, we define
\[\bot\cdot m:=\bot_N \quad\text{ and }\quad\top\cdot m:=m,\]
for every $m\in N$. It is easy to see that this makes $N$ into a $\{\bot,\top\}$-module.
\end{enumerate}
\end{example}

\begin{definition}
    A \emph{homomorphism of modules} between $Q$-modules $M$ and $N$ is a function $f\:M\to N$ such that
\begin{enumerate}
\item $f$ preserves arbitrary joins;

\item $f$ is equivariant, \textit{i.e.},\ for every $\alpha\in Q$ and $m\in M$ we have
\[f(\alpha\cdot m)=\alpha\cdot f(m).\]
\end{enumerate}
\end{definition}

Homomorphisms of modules are in a sense linear maps, thinking about the join as a sum and about $\cdot$ as a scalar multiplication. 

A \emph{submodule} of a $Q$-module $M$  is a subset of $M$ which is closed under arbitrary joins and scalar multiplication $\cdot$.
Submodules of a $Q$-module $M$ are automatically complete lattices and thus $Q$-modules. Furthermore, submodules are equivalently given by injective homomorphisms of modules. Notice that every submodule $N\subseteq M$ needs to contain $\bot_M$, since it is the empty join, and we have $\bot_N=\bot_M$.

The following proposition shows how to construct  $Q$-submodules generated by an element of a $Q$-module. 

\begin{proposition}
Let $M$ be a $Q$-module and let $m\in M$. The submodule $\langle m \rangle$ generated by $m$, i.e.,\ the smallest submodule of $M$ containing $m$, is given by
\[Q\cdot m=\set{\alpha\cdot m}{\alpha\in Q}.\]
\end{proposition}

\begin{proof}
Of course, $Q\cdot m$ needs to be contained in every submodule of $M$ which contains $m$, since submodules are closed under scalar multiplication. So, it suffices to prove that $Q\cdot m$ is a submodule of $M$. Since $\cdot\:Q\times M\to M$ preserves arbitrary joins, we have
\[\bigvee (\alpha_i\cdot m)=(\bigvee \alpha_i)\cdot m\in Q\cdot m.\]
Moreover $Q\cdot m$ is closed under scalar multiplication, as for every $\beta,\alpha\in Q$
\[\beta\cdot (\alpha\cdot m)=(\beta\ast \alpha)\cdot m\in Q\cdot m.\] Thus, $Q\cdot m$ is the submodule of $M$.
\end{proof}

\begin{remark}
In particular, notice that $\langle \bot_M \rangle=\{\bot_M\}$. We will also denote this submodule of $M$ as $0$.

The image $\Im{f}$ of a homomorphism of $Q$-modules $f\:M\to M'$  is a submodule of $M'$. Also, the kernel $\operatorname{Ker}(f)$ of a homomorphism of $Q$-modules $f\:M\to M'$  is a submodule of $M$.
\end{remark}

Aiming at the concept of \emph{primitive quantale}, we need to introduce notions of simple and of faithful quantales.

\begin{definition}
A $Q$-module $M$  is called \emph{simple} if the only submodules of $M$ are the zero submodule $0$ and the whole $M$.
\end{definition}

\begin{proposition}
Let $M$ be a $Q$-module. Then the following are equivalent:
\begin{enumerate}
\item $M$ is simple;

\item For every $m\in M$ with $m\neq \bot_M$, we have $\langle m \rangle=M$.
\end{enumerate}
\end{proposition}

\begin{proof}
The proof of (1)$\Rightarrow$(2) is trivial. To prove (2)$\Rightarrow$(1), let $N\subseteq M$ be a submodule of $M$. If $N$ is not the zero submodule $0$, then it contains an element $n\neq \bot_M$ and it thus also contain the submodule $\langle n \rangle=M$. Therefore, $N=M$.
\end{proof}

\begin{definition}
Let $M$ be a $Q$-module. 
\begin{enumerate}
\item[$\bullet$] $M$ is called \emph{faithful} if for every $\alpha\in Q$ such that $\alpha \cdot m= \bot_M$ for every $m\in M$, we have that $\alpha=\bot_Q$. 

\item[$\bullet$] $M$ is called \emph{strongly faithful} if for every $\alpha,\beta\in Q$ such that $\alpha \cdot m= \beta \cdot m$ for every $m\in M$, we have that $\alpha=\beta$.  
\end{enumerate}
\end{definition}

\begin{remark} \label{weaklyfaithful}
The property of being a strongly faithful $Q$-module  is stronger than the property of being faithful. Indeed, the latter is given by the former taking $\beta=\bot_Q$.
\end{remark}

\begin{proposition}\label{propfaithfulann}
Let $M$ be a $Q$-module. The following facts are equivalent:
\begin{enumerate}
\item $M$ is faithful;

\item The homomorphism of quantales $[\cdot]\:Q\to \End{M}$ associated to the module $M$ has trivial kernel.
\end{enumerate}
\end{proposition}

\begin{proof}
Follows from the definition.
\end{proof}

\begin{definition}
Let $M$ be a $Q$-module. The \emph{annihilator} of $M$ is the kernel of the homomorphism of quantales $[\cdot]\:Q\to \End{M}$ associated to the module $M$.
\end{definition}

\begin{remark}\label{remfaithfulann}
We have thus proved that $M$ is a faithful $Q$-module  if and only if its annihilator is the zero subquantale of $Q$.
Strong faithfulness corresponds to the request that the homomorphism of quantales $[\cdot]\:Q\to \End{M}$ is injective. Since quantales lack differences, in the sense that the join operation $\vee$ of a quantale has no inverses, kernels of homomorphisms of quantales do not fully capture injectivity, unlike what happens in the classical ring theory.
We shall face this obstacle  throughout the paper.
\end{remark}

\section{Primitive quantales}\label{pqt}

The aim of this section is to introduce the concept of primitive quantales. We prove that every primitive ring gives rise to a primitive quantale of ideals. We then show how analogues of the fundamental properties of primitive rings hold in our context as well.

\begin{definition}
A quantale $Q$ is called $\emph{primitive}$ if there exists a faithful simple $Q$-module and
$Q$ is called \emph{strongly primitive} if there exists a strongly faithful simple $Q$-module.
\end{definition}

The following results show how the theory of primitive quantales extends the theory of primitive rings. It is well-known that every ring gives rise to a quantale of ideals. We show that an analogous construction can be performed on modules.

\begin{theorem}\label{theorsubmodules}
Let $M$ be an $R$-module. Then the set $\Sub{M}$ of submodules of $R$ is a module over the quantale $\Idl{R}$ of ideals of $R$. Moreover,
\begin{enumerate}
\item if $M$ is a faithful $R$-module, then $\Sub{M}$ is a faithful module over $\Idl{R}$.

\item If $M$ is a simple $R$-module, then $\Sub{M}$ is a simple module over $\Idl{R}$.
\end{enumerate}
\end{theorem}

\begin{proof}
It is well-known that $\Idl{R}$ is a quantale. We prove that the set $\Sub{M}$ of submodules of $M$ over $R$ can be equipped with a structure of module over $\Idl{R}$. Notice that $\Sub{M}$ is a complete lattice, with the joins being sums of submodules and the meets given by intersections. We can then define a scalar multiplication $\cdot\:\Idl{R}\times \Sub{M}\to \Sub{M}$ sending a pair $(I,N)$ with $I$ an ideal of $R$ and $N$ a submodule of $N$ to the multiplication $I\cdot N\subseteq M$. It is straightforward to see that $\cdot$ preserves (arbitrary) joins and that $(I\cdot J)\cdot N=I\cdot (J\cdot N)$ for  ideals $I,J$ of $R$ and $N$ submodule of $M$. Furthermore, since $R$ is unitary, we have $R\cdot N=N$. From these, we conclude that $\Sub{M}$ is a module over $\Idl{R}$.

To show (1), assume now that $M$ is a faithful $R$-module. We prove that if $I$ is a non-zero ideal of $R$, then there exists a submodule $N$ such that $I\cdot N\neq 0$. Let $a\in I$ with $a\neq 0$. Then, since $M$ is faithful, there exists $m\in M$ such that $a\cdot m\neq 0_M$. As a consequence, the submodule $N=\langle m \rangle$ of $M$ generated by $m$ is such that $I\cdot N\neq 0$. We thus conclude that $\Sub{M}$ is a faithful module over $\Idl{R}$.

Finally, for (2), if $M$ is a simple $R$-module, then $\Sub{M}=\{0,M\}$ is clearly a simple module over $\Idl{R}$.
\end{proof}

\begin{corollary}\label{corollprimitiverings}
Let $R$ be a primitive ring. Then the quantale $\Idl{R}$ of ideals of $R$ is primitive.
\end{corollary}

\begin{proof}
Let $R$ be a primitive ring and $M$ be a faithful simple $R$-module. Then by Theorem \ref{theorsubmodules}, the set $\Sub{M}$ of submodules of $M$ is a faithful simple module over the quantale $\Idl{R}$, and thus, $\Idl{R}$ is a primitive quantale.
\end{proof}

\begin{remark}
Observe that Corollary \ref{corollprimitiverings} provides numerous examples of primitive quantales. Every primitive ring gives rise to a primitive quantale of ideals.
Furthermore, Corollary \ref{corollprimitiverings} also shows how the theory explored in this paper extends the classical theory of primitive rings.
\end{remark}

We now prove that analogues of the fundamental properties of primitive rings hold in our context. For that, we start with a definition.

\begin{definition}\label{defprime}
A quantale $Q$ is called \emph{prime} if for every $\alpha,\beta\in Q$ such that $\alpha\ast \gamma\ast \beta=\bot_Q$ for every $\gamma\in Q$, we have that either $\alpha=\bot_Q$ or $\beta=\bot_Q$.
\end{definition}

\begin{proposition}
Every primitive quantale is prime.
\end{proposition}

\begin{proof}
Let $Q$ be a primitive quantale and let $\alpha,\beta\in Q$ such that $\alpha\neq \bot_Q\neq \beta$. We prove that there exists some $\gamma\in Q$ such that $\alpha\ast\gamma\ast\beta\neq \bot_Q$. Let $M$ be a faithful simple module over $Q$. Since $M$ is faithful, there exists $m\in M$ such that $\alpha\cdot m\neq \bot_M$ and $n\in M$ such that $\beta\cdot n\neq \bot_M$. Since $M$ is simple,
\[Q\cdot (\beta\cdot n)=\langle \beta\cdot n \rangle=M.\]
So, there exists $\gamma\in Q$ such that $m=\gamma\cdot (\beta\cdot n)$. Then
\[(\alpha\ast \gamma\ast \beta)\cdot n=\alpha\cdot (\gamma\cdot (\beta\cdot n))=\alpha\cdot m\neq \bot_M,\]
whence $\alpha\ast\gamma\ast\beta\neq \bot_Q$.
\end{proof}

\begin{definition}
Let $Q$ be a quantale. A \emph{left ideal} of $Q$ is a submodule of $Q$ seen as a $Q$-module.
Equivalently, a left ideal of $Q$ is a subset $I$ of $Q$ which is closed under arbitrary joins and is such that for every $\alpha\in Q$ and $\iota\in I$ we have that $\alpha\ast \iota\in I$.
\end{definition}

\begin{proposition}
Every prime quantale with a minimal left ideal is primitive.
\end{proposition}

\begin{proof}
Let $Q$ be a prime quantale with a minimal left ideal $I$. Then $I$ is a $Q$-module. Moreover, minimality of $I$ guarantees that $I$ is a simple $Q$-module. Indeed, every submodule of $I$ needs to be a submodule of $Q$ as well, and thus, an ideal of $Q$ contained in $I$. It remains to prove that $I$ is a faithful $Q$-module. Let $\alpha\in Q$ such that for every $\iota\in I$ we have $\alpha\ast \iota=\bot_Q$. We show that $\alpha=\bot_Q$. Let $\lambda\in I$ with $\lambda\neq \bot_Q$ (one such $\lambda$ exists because $I$ is a minimal left ideal). For every $\gamma\in Q$, we need to have that
\[\alpha\ast \gamma\ast \lambda=\bot_Q,\]
since $\gamma\ast \lambda\in I$. Since $Q$ is prime, we can then conclude that either $\alpha=\bot_Q$ or $\lambda=\bot_Q$. But we have assumed $\lambda\neq \bot_Q$, so $\alpha=\bot_Q$.   
\end{proof}

\begin{definition}
A \emph{division quantale} is a quantale $Q$ such that for every $\alpha\in Q$ with $\alpha\neq \bot_Q$, there exists $\alpha^{-1}\in Q$ such that $\alpha\ast \alpha^{-1}=e=\alpha^{-1}\ast \alpha$. That is, every non-zero element has a (two-sided) inverse with respect to the product $\ast$.  

A \emph{field quantale} is a commutative division quantale.
\end{definition}

\begin{example}\label{exatruthvaluesfield}
The quantale of truth values $\{\bot,\top\}$ is a field quantale. Indeed, it is commutative because the product is given by conjunctions. And $\top$ is invertible with respect to the conjunction, with its inverse being itself.
\end{example}

\begin{proposition}
Every commutative, strongly primitive quantale is a field quantale.
\end{proposition}

\begin{proof}
Let $Q$ be a commutative strongly primitive quantale and let $M$ be a strongly faithful simple $Q$-module. Consider  an  $m\in M$ with $m\neq \bot_M$. Notice that one such $m$ exists because $M$ is faithful. We claim that for every $\alpha,\beta\in Q$, if $\alpha\cdot m=\beta\cdot m$ then $\alpha=\beta$. Indeed, assume $\alpha\cdot m=\beta\cdot m$, and let $n\in M$ arbitrary. We show that also $\alpha\cdot n=\beta\cdot n$ holds. Since $M$ is simple, $Q\cdot m=\langle m \rangle =M$, and thus, there exists $\gamma\in Q$ such that $n=\gamma\cdot m$. Then
\[\alpha\cdot n=\alpha\cdot (\gamma\cdot m)=(\alpha\ast \gamma)\cdot m=(\gamma\ast \alpha)\cdot m=\gamma\cdot (\alpha\cdot m)=\gamma\cdot (\beta\cdot m)=(\gamma\ast \beta)\cdot m=(\beta\ast \gamma)\cdot m=\beta\cdot (\gamma\cdot m)=\beta\cdot n.\]
Notice that we have used the commutativity of $\ast$. Since $M$ is strongly faithful, we conclude that $\alpha=\beta$.

Consider now $\alpha\in Q$ with $\alpha\neq \bot_Q$. We prove that $\alpha$ has a two-sided inverse with respect to $\ast$. Thanks to what we proved above, $\alpha\cdot m\neq \bot_M$, because $\bot_Q\cdot m=\bot_M$ and $\alpha\neq\bot_Q$. Since $M$ is simple, $Q\cdot (\alpha\cdot m)=\langle \alpha\cdot m \rangle=M$. So, there exists $\delta\in Q$ such that $m=\delta\cdot (\alpha\cdot m)$. Then
\[(\delta\ast \alpha)\cdot m=\delta\cdot (\alpha\cdot m)=m=e\cdot m.\]
Using again what we proved above, we conclude that $\delta\ast \alpha=e$. Since $Q$ is commutative, also $\alpha\ast \delta=e$ holds, and hence, $Q$ is a field quantale.
\end{proof}

\section{A  density theorem}\label{dtq}

In this section, we present a density theorem for primitive quantales in the spirit of the fundamental Jacobson Density Theorem from the classical theory of primitive rings.

\begin{proposition}\label{propendqm}
Let $M$ be a $Q$-module. The set $\End[Q]{M}$ of endomorphisms of the $Q$-module $M$ can be endowed with the structure of a quantale, with the product given by composition.

Furthermore, $M$ is a module over the quantale $\End[Q]{M}$, with the scalar multiplication given by evaluation.
\end{proposition}

\begin{proof}
By Remark \ref{remend}, the set $\End{M}$ of join-preserving functions $M\to M$ can be endowed with the structure of a quantale, with the joins defined pointwise and the product given by composition. $\End[Q]{M}$ is then a subquantale of $\End{M}$, since it contains the identity of $M$ and it is closed under arbitrary joins and composition.
    
We now prove that $M$ is a module over the quantale $\End[Q]{M}$, with scalar multiplication given by evaluation. That is, we define $f\cdot m$ to be $f(m)\in M$, for every $f\in \End[Q]{M}$ and $m\in M$. Since joins in $\End[Q]{M}$ are defined pointwise and every $f\in\End[Q]{M}$ preserves joins, $\cdot$ preserves arbitrary joins in both components. It is then easy to see that the evaluation function satisfies the required axioms of the module, using that $\ast$ is given by composition. Hence, $M$ is an $\End[Q]{M}$-module.
\end{proof}

For the density theorem, we need a notion of basis of a $Q$-module.

\begin{definition}
Let $M$ be a $Q$-module. A \emph{weak basis} of $M$ (as a $Q$-module) is a family $\{t_i\}_{i\in I}$ of elements of $M$ such that every $m\in M$ can be written in a unique way as an arbitrary join $\bigvee_{j\in J} (\alpha_j\cdot t_j)$ with $J\subseteq I$ and $\alpha_j\in Q$ with $\alpha_j\neq \bot_Q$ for every $j\in J$.

The family
$\{t_i\}_{i\in I}$ is called a \emph{basis} of $M$ (as a $Q$-module) if the joins $\bigvee_{j\in J} (\alpha_j\cdot t_j)$ above are all finite.
\end{definition}

\begin{proposition}
\label{propbasis}
Let $M$ be a $Q$-module with a weak basis $\{t_i\}_{i\in I}$. Every assignment $t_i\mapsto m_i$ of the elements of the weak basis to other elements of $M$ can be extended to a homomorphism of modules $M\to M$.
\end{proposition}

\begin{proof}
Every $n\in M$ can be written in a unique way as an (arbitrary) join $\bigvee_{j\in J}(\alpha_j\cdot t_j)$ with $J\cont I$, $\alpha_j\in Q$ and $\alpha_j\neq \bot_Q$, for every $j\in J$. We then send $n$ to $\bigvee_{j\in J}(\alpha_j\cdot m_j)$. It is straightforward to see that this defines a homomorphism of modules $M\to M$ that extends the original assignment.
\end{proof}

\begin{remark}
\label{remlackbasis}
Because of the lack of differences in a quantale, it is not true that every module over a division quantale has a basis, not even a weak basis. Indeed, Example \ref{exalackbasis} below is a counterexample. 
\end{remark}

\begin{remark}
\label{remeverylatticemoduleoverdivision}
By Example \ref{fex}(\ref{exaeverylatticeismodule}), every complete lattice can be seen as a module over the quantale $\{\bot,\top\}$ of truth values. Moreover, by Example \ref{exatruthvaluesfield}, the quantale of truth values is a division quantale (in fact, a field quantale). A basis for a complete lattice $M$, seen as a $\{\bot,\top\}$-module, is a set $\{t_i\}_{i\in I}$ of elements of $M$ such that every $m\in M$ can be written in a unique way as an arbitrary join $\bigvee_{j\in J} t_j$ with $J\subseteq I$.
\end{remark}

\begin{example}\label{exalackbasis}
Consider the complete lattice $\Idl{\mathds{Z}}$ of ideals of the ring of integers. We prove that $\Idl{\mathds{Z}}$ has no weak basis (and thus no basis either) as a $\{\bot,\top\}$-module, despite the fact that $\{\bot,\top\}$ is a field quantale. Assume by contradiction that $\Idl{\mathds{Z}}$ has a weak basis, and consider an ideal $(n)$ of the weak basis with $n\neq 1$ (one such ideal exists because, e.g.\ $(2)$ needs to be written as a sum of ideals of the weak basis). Then $(n^2)\subset (n)$ can be written in a unique way as a (possibly infinite) sum of ideals $\{T_j\}_{j\in J}$ of the weak basis. But since $(n)=(n)+(n^2)$, all the ideals $T_j$ need to appear in representing $(n)$ as a sum of ideals of the weak basis. This is a contradiction, as also $(n)=(n)$ is a representation of $(n)$ as a sum of ideals of the weak basis, and such a representation should be unique.
\end{example}

In light of Remark \ref{remlackbasis}, we shall sometimes require the existence of a basis as an explicit assumption on our quantales.

\begin{proposition}
\label{propatomic}
Every atomic frame is a module with a weak basis over the field quantale $\{\bot,\top\}$, with the weak basis given by the atoms.
\end{proposition}

\begin{proof}
By Remark \ref{remeverylatticemoduleoverdivision}, every complete lattice $L$ can be seen as a module over the field quantale $\{\bot,\top\}$ of truth values. Assume that $L$ is an atomic frame. We prove that the set $\{a_i\}_{i\in I}$ atoms of $L$ is a weak basis of $L$, as a $\{\bot,\top\}$-module. Since $L$ is atomic, every element $\ell\in L$ can be written as an (arbitrary) join of atoms. Thanks to Remark \ref{remeverylatticemoduleoverdivision}, it just remains to prove that such writings are unique. Let $\ell\in L$ and consider two writings
\[\bigvee_{j\in J}a_j=\ell=\bigvee_{k\in K}a_k\]
of $\ell$ as a join of atoms, with $J,K\subseteq I$. We show that $J=K$. So consider $\overline{j}\in J$. Since $L$ is a frame,
\[\bigvee_{j\in J}(a_{\overline{j}}\wedge a_j)=a_{\overline{j}}\wedge\left(\bigvee_{j\in J}a_j\right)=a_{\overline{j}}\wedge\left(\bigvee_{k\in K}a_k\right)=\bigvee_{k\in K}(a_{\overline{j}}\wedge a_k).\]
Notice now that if $i\in I$ is such that $i\neq \overline{j}$, then $a_{\overline{j}}\wedge a_i=\bot$. Indeed, $a_{\overline{j}}\wedge a_i\leqslant a_{\overline{j}}$. Since $a_{\overline{j}}$ is an atom, $a_{\overline{j}}\wedge a_i$ is then either $\bot$ or $a_{\overline{j}}$. But if it were equal to $a_{\overline{j}}$, we would have $a_{\overline{j}}\leqslant a_i$, which is a contradiction because $a_i$ is an atom. We conclude that $\bigvee_{j\in J}(a_{\overline{j}}\wedge a_j)=a_{\overline{j}}$ and that one of the $k$'s needs to be $\overline{j}$. Whence $J\subseteq K$. Analogously, $K\subseteq J$.
\end{proof}

\begin{corollary}
Every quantale $\P{X}$ of the power set of a set $X$ is a module with a weak basis over the field quantale $\{\bot,\top\}$, with the weak basis given by singletons.
\end{corollary}

\begin{proof}
Straightforward application of Proposition \ref{propatomic}. Notice that $\P{X}$ is an atomic frame, and its atoms are precisely the singletons $\{x\}$ with $x\in X$.
\end{proof}

Aiming at a density result, we define a notion of (weakly) dense quantale of linear operators.

\begin{definition}\label{defdense}
Let $M$ be a $Q$-module. 
\begin{enumerate}
\item[$\bullet$] A \emph{weakly dense quantale of linear operators of $M$} is a subquantale $D$ of $\End[Q]{M}$ such that for every $t,m\in M$ with $t\neq \bot_M$ there exists $f\in D$ such that $f(t)=m$. 

\item[$\bullet$] Suppose $M$ has a basis. A \emph{dense quantale of linear operators of $M$} is a subquantale $D$ of $\End[Q]{M}$ such that, for every finite set $\{t_1,\ldots t_n\}$ of elements of $M$ completable to a basis and every other set $\{m_1,\ldots, m_n\}$ of elements of $M$, there exists $f\in D$ such that $f(t_i)=m_i$ for every $i\in \{1,\ldots,n\}$.
\end{enumerate}  
\end{definition}

\begin{remark}\label{remnoteveryelemcompletable}
Not every element of a module with a basis can be completed to a basis. As a counterexample, consider the $\{\bot,\top\}$-module $\P{\{1,2,3\}}$ of the power set of the set $\{1,2,3\}$. The element $\{1,2\}$ cannot be completed to a basis. Indeed, the only way to write $\{1\}$ as a join in $\P{\{1,2,3\}}$ is to use $\{1\}$ itself, and analogously for $\{2\}$. But then $\{1,2\}$ would be written both as $\{1,2\}$ and as $\{1\}\vee \{2\}$.
\end{remark}

Nevertheless, the following  holds.
 
\begin{proposition}\label{propdenseimpliesweaklydense}
Let $M$ be a module with a basis over a division quantale $Q$. Every dense quantale of linear operators of $M$ is also a weakly dense quantale of linear operators of $M$.
\end{proposition}

\begin{proof}
Denote by $\{t_i\}_{i\in I}$, the basis of $M$ as a $Q$-module. Let $D\subseteq \End[Q]{M}$ be a dense quantale of linear operators of $M$. Let  $t,m\in M$ with $t\neq \bot_M$. We prove that there exists $f\in D$ such that $f(t)=m$. We can write $t$ as a finite linear combination of the basis:
\[t=\bigvee_{j\in J}(\alpha_j\cdot t_j),\]
with $J\cont I$, $\alpha_j\in Q$ and $\alpha_j\neq \bot_Q$ for every $j\in J$. Consider a fixed $\overline{j}\in J$. Of course, $\{t_j\}_{j\in J}$ is completable to a basis of $M$. Furthermore, there exists an inverse $\alpha_{\overline{j}}^{-1}$ of $\alpha_{\overline{j}}$, as $Q$ is a division quantale. Since $D$ is a dense quantale of linear operators of $M$, there exists $f\in D$ such that $f(t_{\overline{j}})=\alpha_{\overline{j}}^{-1}\cdot m$ and $f(t_j)=\bot_M$ for every $j\in J$ with $j\neq \overline{j}$. Then
\[
f(t)=f(\bigvee_{j\in J}(\alpha_j\cdot t_j))=\bigvee_{j\in J}(\alpha_j\cdot f(t_j))=\alpha_{\overline{j}}\cdot (\alpha_{\overline{j}}^{-1}\cdot m)=m. \qedhere
\]
\end{proof}

\begin{proposition}
Let $M$ be a $Q$-module with a weak basis. Then  $\End[Q]{M}$ is a dense quantale of linear operators of $M$.
\end{proposition}

\begin{proof}
Let $\{t_1,\ldots t_n\}$ be a finite set of elements of $M$ completable to a basis $B$ and let $\{m_1,\ldots, m_n\}$ be another set of elements of $M$. Thanks to Proposition \ref{propbasis}, the assignment $t_j\mapsto m_j$ for every $j\in \{1,\ldots,n\}$ and $t\mapsto \bot_M$ for every $t$ in the basis $B$ with $t$ different from the $t_j$'s, can be extended to a homomorphism of modules $M\to M$. 
\end{proof}

\begin{remark}
In the classical definition of a dense ring of linear operators, one asks $\{t_1,\ldots,t_n\}$ to be linearly independent, as opposed to our request that it is completable to a basis, in Definition \ref{defdense}. Of course, for modules over a division ring, the two conditions are equivalent. But they are not equivalent in the context of modules over a quantale, as shown by Remark \ref{remnoteveryelemcompletable}. The following example shows the importance of the request that $\{t_1,\ldots,t_n\}$ is completable to a basis in Definition \ref{defdense}.
\end{remark}

\begin{example}
In the $\{\bot,\top\}$-module $\P{\{1,2,3\}}$ of the power set of the set $\{1,2,3\}$, the elements $\{1,2\}$ and $\{1\}$ are linearly independent in the sense that $\langle \{1,2\} \rangle\cap \langle \{1\} \rangle=\{\bot\}$. Indeed
$$\langle \{1,2\} \rangle=\{\bot,\top\}\cdot \{1,2\}=\{\emptyset,\{1,2\}\}\quad \text{ and } \quad \langle \{1\} \rangle=\{\emptyset,\{1\}\}.$$
But there exists no homomorphism of modules $f\:\P{\{1,2,3\}}\to \P{\{1,2,3\}}$ over $\{\bot,\top\}$ such that $f(\{1,2\})=\{1\}$ and $f(\{1\})=\{2\}$. Indeed, $f$ would need to satisfy
$$f(\{1,2\})=f(\{1\}\vee \{2\})=f(\{1\})\vee f(\{2\})=\{2\}\vee f(\{2\})\supseteq \{2\}.$$
\end{example}

We can now present a density theorem for primitive quantales.

\begin{theorem}
Let $M$ be a $Q$-module. Then the following hold:
\begin{enumerate}
\item Every weakly dense quantale of linear operators of a module $M$  is strongly primitive (and thus primitive).
In particular, every dense quantale of linear operators of a module $M$ with a basis over a division quantale is strongly primitive.
        
\item Every strongly primitive quantale is isomorphic to a weakly dense quantale of linear operators of a module $M$.
\end{enumerate}
\end{theorem}

\begin{proof}
(1) Let $D\cont \End[Q]{M}$ be a weakly dense quantale of linear operators of a $Q$-module $M$. We prove that $M$ is a strongly faithful simple module over $D$, so that $D$ is strongly primitive. By Proposition \ref{propendqm}, $M$ is a module over $\End[Q]{M}$. Then $M$ is also a module over the subquantale $D\cont \End[Q]{M}$. In order to prove that $M$ is strongly faithful, consider $f,g\:M\to M$ in $D$ such that $f\cdot m=g\cdot m$ for every $m\in M$. Since $f\cdot m=f(m)$ and $g\cdot m=g(m)$, we conclude that $f=g$. Finally, we show that $M$ is simple as a $D$-module. Let $m\in M$ with $m\neq \bot_M$. Then \[\langle m \rangle =D\cdot m=M.\] Indeed, since $D$ is a weakly dense quantale of linear operators of $M$, for every $n\in M$ there exists $f\in D$ such that $f(m)=n$. And thus $n=f\cdot m\in D\cdot m$. We conclude that $M$ is a strongly faithful simple module over $D$.
Thanks to Proposition \ref{propdenseimpliesweaklydense}, we then have that every dense quantale of linear operators of a module $M$ with a basis over a division quantale is strongly primitive.

(2) Let $Q$ be a strongly primitive quantale and let $M$ be a strongly faithful simple $Q$-module. We consider the quantale $\End[Q]{M}$ of endomorphisms of modules from $M$ to $M$ over $Q$. By Proposition \ref{propendqm}, $M$ is a module over $\End[Q]{M}$. Moreover, the homomorphism of quantales $[\cdot]\:Q\to \End{M}$ associated to $M$ factors through the subquantale $\End[{\End[Q]{M}}]{M}$ of $\End{M}$. Indeed, for every $\alpha\in Q$, the join-preserving function $[\cdot](\alpha)=\alpha\cdot -\:M\to M$ is a homomorphism of ${\End[Q]{M}}$-modules, as for every $f\in \End[Q]{M}$ and $m\in M$ we have $\alpha\cdot (f(m))=f(\alpha\cdot m)$. Call $\psi$ the resulting homomorphism of quantales $Q\to \End[{\End[Q]{M}}]{M}$. By Proposition \ref{propfaithfulann} and Remark \ref{remfaithfulann}, since $M$ is strongly faithful, $[\cdot]$ is injective. And then also $\psi$ needs to be injective. Therefore, $\psi$ exhibits an isomorphism of quantales between $Q$ and
\[\Im{\psi}=\set{\alpha\cdot -\:M\to M}{\alpha\in Q}\cont\End[{\End[Q]{M}}]{M}.\]
It remains to prove that $\Im{\psi}$ is a weakly dense quantale of linear operators of $M$, over $\End[Q]{M}$. So let $t,m\in M$ with $t\neq \bot_M$. Since $M$ is a simple $Q$-module, $Q\cdot t=\langle t \rangle =M$. Then there exists $\alpha\in Q$ such that $\alpha\cdot t=m$. Whence $\alpha\cdot -\in \Im{\psi}$ is a homomorphism of modules $M\to M$ over $\End[Q]{M}$ such that $(\alpha\cdot -)(t)=m$.
\end{proof}

\begin{remark}
In the classical case, if $M$ is a simple $R$-module, then $\End[R]{M}$ is a division ring. This is because every non-zero linear map $M\to M$ has to be bijective since both the kernel and the image are submodules of $M$, and $M$ is simple. In the case of quantales, we still know that every homomorphism of modules $M\to M$ with $M$ simple has to be surjective; however, the kernel being zero does not guarantee injectivity.
\end{remark}

\end{document}